\documentclass{amsart}

\newtheorem{theorem}{Theorem}[section]
\newtheorem{lemma}[theorem]{Lemma}

\theoremstyle{definition}

\newtheorem{externaltheorem}{Theorem}

\newtheorem{corollary}[theorem]{Corollary}
\newtheorem{proposition}[theorem]{Proposition}

\usepackage{graphicx}
\usepackage{amssymb}
\usepackage{amsrefs}

\usepackage{enumerate}

\theoremstyle{remark}

\numberwithin{equation}{section}

\begin{document}

\title[Avoiding infinite arithmetic progressions]{Large subsets of Euclidean space avoiding infinite arithmetic progressions}

\author{Laurestine Bradford}
\address{Department of Linguistics, McGill University, Montreal, Quebec, H3A 1A7, Canada and Centre for Research on Brain, Language and Music, Montreal, Quebec, H3G 2A8, Canada}
\email{laurestine.bradford@mail.mcgill.ca}
\thanks{The first author is supported by a CRBLM Graduate Student Stipend. The CRBLM is funded by the Government of Quebec via the Fonds de Recherche Nature et Technologies and Société et Culture}

\author{Hannah Kohut}
\address{Department of Mathematics, University of British Columbia, Vancouver, British Columbia, V6T 1Z2, Canada}
\email{kohut@math.ubc.ca}
\thanks{The second author is supported by NSERC Discovery Grants 22R81123 and 22R00756} 

\author{Yuveshen Mooroogen}
\address{Department of Mathematics, University of British Columbia, Vancouver, British Columbia, V6T 1Z2, Canada}
\email{yuveshenm@math.ubc.ca}
\thanks{The third author is supported by NSERC Discovery Grant GR010263}


\date{April 19, 2023}

\begin{abstract}
It is known that if a subset of $\mathbb{R}$ has positive Lebesgue measure, then it contains arbitrarily long finite arithmetic progressions. We prove that this result does not extend to infinite arithmetic progressions in the following sense: for each $\lambda$ in $[0,1)$, we construct a subset of $\mathbb{R}$ that intersects every interval of unit length in a set of measure at least $\lambda$, but that does not contain any infinite arithmetic progression.
\end{abstract}

\maketitle

\section{Introduction}

A \textit{finite arithmetic progression of length} $k$ is a set of the form $\{x, x + \Delta, x + 2\Delta, ..., x+(k-1) \Delta\}$ for some $x$ in $\mathbb{R}^n$ and $\Delta \neq 0$ in $\mathbb{R}^n$. It follows from the Lebesgue density theorem that every Lebesgue measurable subset $S$ of $\mathbb{R}^n$ with positive Lebesgue measure must contain finite arithmetic progressions of length $k$ for every $k$ in $\mathbb{N}$.

However, the same is not true of \textit{infinite} arithmetic progressions. We say that a subset $S$ of $\mathbb{R}^n$ contains an \textit{infinite arithmetic progression} or simply an \textit{arithmetic progression} if there exists a point $x$ in $S$ and a vector $\Delta \neq 0$ in $\mathbb{R}^n$ such that the set $x + \Delta \mathbb{N} = \{x + n\Delta : n \in \mathbb{N}\}$ is contained in $S$. We call $\Delta$ the \textit{gap length} or \textit{gap} of the progression. In this paper, we focus on arithmetic progressions with positive gap length, although symmetrical arguments apply to progressions with negative gap length. It is plain that a set of positive Lebesgue measure need not contain an infinite arithmetic progression; indeed, no bounded set contains any infinite arithmetic progression.

Sets of infinite measure may also fail to contain arithmetic progressions. For example, the set $S = \bigcup_{n \in \mathbb{N}} [n^2,n^2 + 1]$ has infinite measure but contains no infinite arithmetic progression. To see this, observe that while the set is made up of infinitely many intervals, the spaces between consecutive intervals grow without bound. On the other hand, the gap length in any given arithmetic progression is fixed, so no infinite arithmetic progression can lie completely within $S$. 

It is natural to ask whether a disjoint union of intervals must contain arithmetic progressions if it is not allowed to have arbitrarily large spaces between intervals. In this paper, we investigate the following question, proposed by Joe Repka to the third author \cite{Joe}: for each $\lambda \in [0,1)$, does there exist a subset $S$ of $\mathbb{R}$ such that $S$ intersects every interval of unit length in a set of measure at least $\lambda$, but $S$ does not contain any arithmetic progression?

We prove the following constructive result, which leads to an affirmative answer to Repka's question.

\begin{theorem}\label{th:mainresult}
For every $N \geq 1$ in $\mathbb{N}$, there exists a subset $S = S(N)$ of $\mathbb{R}$ such that 
\begin{align*}
    \vert S \cap [m,m+1] \vert = 1 - \frac{1}{N}
\end{align*}
for every $m$ in $\mathbb{Z}$, but that does not contain any arithmetic progression $x + \Delta\mathbb{N}$ for any $x$ in $S$ and $\Delta > 0$ in $\mathbb{R}$.
\end{theorem}

We have recently become aware of two preprints, \cite{KandP} and \cite{newkeleti}, that were inspired by our above result. In \cite{KandP}, Kolountzakis and Papageorgiou prove the following theorem, which extends Theorem \ref{th:mainresult} to a larger class of sequences. 

\begin{externaltheorem}[Kolountzakis--Papageorgiou, Theorem 1.1 from \cite{KandP}]\label{th:kandp}
    Let $\mathbb{A} = \{a_n : n \in \mathbb{N}\}$ be a sequence of real numbers such that (i) $a_0 = 0$, (ii) $a_{n+1} - a_n \geq 1$ for all $n$ in $\mathbb{N}$, and (iii) $\log{a_n} = o(n)$. For each $\lambda \in [0,1)$, there exists a subset $S$ of $\mathbb{R}$ such that $\vert S \cap [m,m+1] \vert \geq \lambda$ for all $m$ in $\mathbb{Z}$, but that does not contain any affine copy of $\mathbb{A}$.
\end{externaltheorem}

In \cite{newkeleti}, Burgin, Goldberg, Keleti, MacMahon, and Wang show that the condition on the size of the sets $S$ in both Theorem \ref{th:mainresult} and Theorem \ref{th:kandp} can be slightly improved.

\begin{externaltheorem}[Burgin--Goldberg--Keleti--MacMahon--Wang, Theorem 2.4 from \cite{newkeleti}]\label{th:burgin}
    Let $\mathbb{A} = \{a_n : n \in \mathbb{N}\}$ be a sequence of positive real numbers such that (i) $a_{n+1} - a_n \geq 1$ for all $n$ in $\mathbb{N}$, and (ii) $\log{a_n} = o(n)$. There exists a subset $S$ of $\mathbb{R}$ such that $\lim_{m \to \infty} \vert S \cap [m,m+1] \vert = 1$, but that does not contain any affine copy of $\mathbb{A}$.
\end{externaltheorem}
The main step leading to this improved result is Lemma 2.3 in \cite{newkeleti}.
 
The techniques employed in the proof of Theorem \ref{th:kandp} differ from the ones we use to prove Theorem \ref{th:mainresult}. In particular, Kolountzakis and Papageorgiou give a probabilistic proof of their result, whereas we give an explicit elementary construction broadly based on equidistribution of sequences on $[0,1)$. Our approach has applications beyond the context of the present problem, which we intend to develop in subsequent work.

Theorem \ref{th:mainresult}, Theorem \ref{th:kandp}, and Theorem \ref{th:burgin} are examples of \textit{avoidance} results, where one shows that a set may be large (in some quantifiable sense) without necessarily containing any affine copy of a prescribed set. Avoidance problems have been extensively studied both in the discrete setting and in the continuum. We identify a few salient results below.

In the discrete setting, Behrend shows in \cite{Behrend} that for any $\epsilon > 0$ and all large enough positive integers $M$, there exists a subset $S$ of $\{0, 1, \ldots, M - 1\}$ with $\#(S) > M^{1 - \epsilon}$ that does not contain any three-term arithmetic progression. Wagstaff, in his paper \cite{Wag}, shows that for every two real numbers $a$ and $b$ with $0 \leq a \leq b \leq 1$, there exists an increasing sequence $S$ of the natural numbers with lower density 
\begin{align*}
    \text{\underbar{$\delta$}}(S) = \liminf_{n \to \infty} \frac{\#(S \cap [1,n])}{n} = a
\end{align*}
and upper density
\begin{align*}
    \bar{\delta}(S) = \limsup_{n \to \infty} \frac{\#(S \cap [1,n])}{n} = b,
\end{align*}
but that does not contain any infinite arithmetic progression. In particular, when $a = b = 1$, Wagstaff's result gives a set of density $1$ that does not contain infinite arithmetic progressions. Behrend and Wagstaff's results both provide a counterpoint to Szemerédi's theorem, which says that every subset $S \subseteq \mathbb{N}$ with positive upper density 
must contain finite arithmetic progressions of length $k$ for every $k$ in $\mathbb{N}$ \cite{Szemeredi}.

In the continuum, the Erd\H{o}s similarity problem asks: given a sequence $\{x_n\} \to 0$, does there exist a subset $S$ of the real line that has positive Lebesgue measure and that does not contain any affine copy of this sequence? This question has been answered in the affirmative for many slowly-decaying sequences, but remains open for many simple examples, such as $\{2^{-n}\}_{n = 0}^\infty$. See \cite{Svetic} for a detailed survey of the progress on this problem up to 2000.

Still in the continuum, but now in the fractal regime, there is a large body of work concerned with constructing sets of large Hausdorff or Fourier dimension that avoid affine copies of prescribed sets. For example, in \cite{Keleti},  Keleti constructs a compact subset of \(\mathbb{R}\) that has full Hausdorff dimension but that does not contain any 3-term arithmetic progression. See also the work of Denson, Pramanik, and Zahl \cite{DensonPramanikZahl}, Fraser and Pramanik \cite{FraserPramanik}, Maga \cite{Maga}, Máthé \cite{Mathe}, Shmerkin \cite{Pablo}, and Yavicoli \cite{Alexia}.

Throughout this paper, we write $\mathbb{N}$ for the set of natural numbers including zero. For any real number $x$, we denote by $\lceil x \rceil$ the smallest integer that is no less than $x$, and by $\lfloor x \rfloor$ the largest integer that is no greater than $x$. We also define the \textit{fractional part} of $x$ to be the real number $\langle x \rangle$ such that $\langle x \rangle = x - \lfloor x \rfloor$. For any subset $A$ of $\mathbb{R}$, we denote by $\#A$ the cardinality of $A$. If $A$ is measurable, we write $\vert A \vert$ for its Lebesgue measure.

\section{Proof of Theorem 1.1.}\label{sec:proof}

\subsection{Construction of the set $S$.}

In what follows, we define $S$ as the disjoint union of a family of sets $\{S_m\}$ such that $S_m \subset [m,m+1)$ for every $m$ in $\mathbb{N}$.

Fix an $N \geq 1$ in $\mathbb{N}$, and divide the interval $[0,1)$ into $N$ disjoint half-open subintervals $Q_0, \ldots, Q_{N-1}$, each of measure $1/N$. More precisely, each $Q_i$ is the interval $[{i}/{N},({i+1})/{N})$. Then, for each $i$ in $\{0, \ldots, N -1\}$, let $R_i = [0,1) \setminus Q_i$. Each $R_i$ has measure $1 - 1/N$. See Figure \ref{fig1} for the special case $N = 3$.
\vspace{-1em}
\begin{figure}[ht]
\centering
\includegraphics[width=\textwidth]{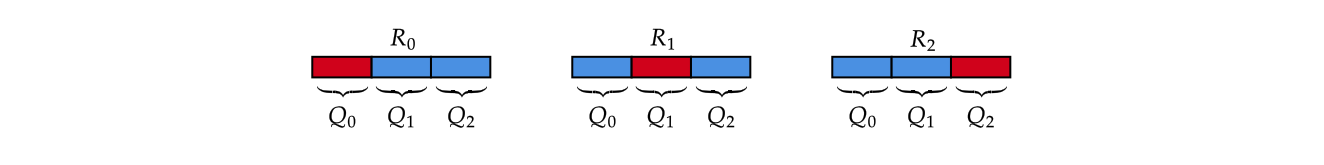}
\vspace{-2em}
\caption{The sets $R_i$ when $N=3$. The deleted intervals $Q_i$ are shaded in red, and the remaining intervals are shaded in blue.}
\label{fig1}
\end{figure}

We construct $S$ out of integer translates of the sets $R_i$. To do so, we introduce the sequence of integers $\{\beta_k\}_{k = 0}^\infty$ where $\beta_k$ is given by the geometric sum
\begin{align*}
    \beta_k = \sum_{j = 0}^{k - 1} (N + 1)^j.
\end{align*}
We also require the translation maps $\tau_m : \mathbb{R} \to \mathbb{R}$, defined by $\tau_m(x) = x + m$.

For each $m \in \mathbb{N}$, define $S_m$ as follows. Let $S_0 = R_0$. For $m > 0$, fix $k$ such that $\beta_{k} \leq m < \beta_{k+1}$ and let $S_m = \tau_m(R_i)$, where $i \equiv k\ (\operatorname{mod} N)$.

For each $m \in \mathbb{Z}$ with $m < 0$, set $S_m = \tau_{2m}(S_{|m|})$. Finally, let
\begin{align*}
    S = \bigcup_{m \in \mathbb{Z}} S_m.
\end{align*}

Observe that $S \cap [m,m+1) = S_m$ for every $m$ in $\mathbb{Z}$. Since translations preserve Lebesgue measure, we therefore have that
\begin{align*}
    \vert S \cap [m,m+1]\vert = 1 - \frac{1}{N},
\end{align*}
for every $m$ in $\mathbb{Z}$, as required by Theorem \ref{th:mainresult}.

\begin{figure}
\centering
\includegraphics[width=\textwidth]{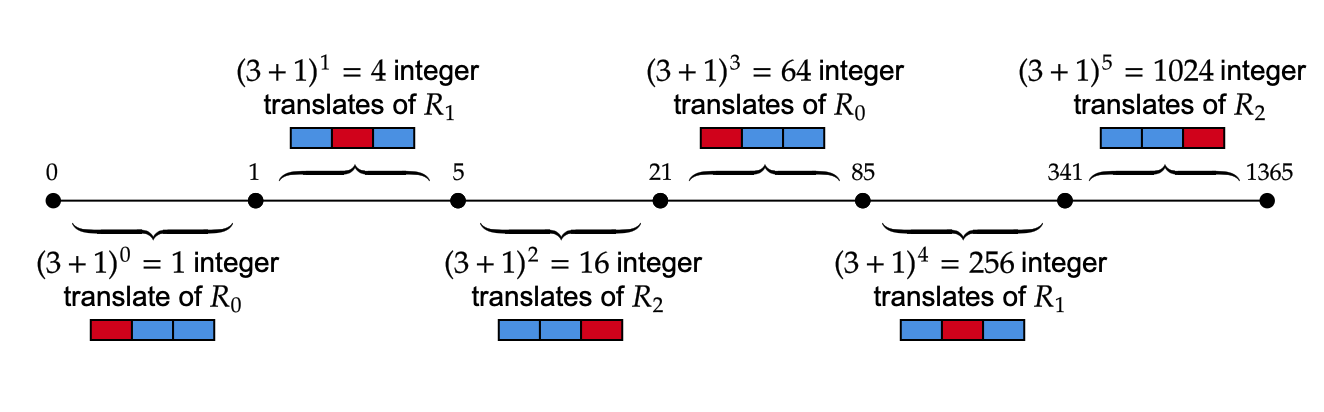}
\vspace{-2em}
\caption{Part of the set $S$ when $N = 3$.}
\label{fig:setS}
\end{figure}

In Figure \ref{fig:setS}, we illustrate part of the set $S$ when $N=3$. We see that the first $(N+1)^0$ intervals $S_m$ are translates of $R_0$, the next $(N+1)^1$ are translates of $R_1$, the next $(N+1)^2$ are translates of $R_2$, and so on. After we reach $(N+1)^{N-1}$ translates of $R_{N-1}$, the next $(N+1)^{N}$ intervals are translates of $R_0$, and we continue cycling through the sets $R_i$ with each successive block of $(N+1)^j$ intervals.

\subsection{Avoiding arithmetic progressions with rational gap length.} 

We show that $S$ does not contain any arithmetic progression whose gap length is a rational number. To do so, it suffices to check that $S$ avoids arithmetic progressions with integer gap length. Indeed, suppose that $S$ contains the arithmetic progression $x + \Delta\mathbb{N}$, where $\Delta = p / q$ for some integers $p$ and $q$. Then $S$ must contain the subset $x + \Delta(q\mathbb{N})$ of $x + \Delta\mathbb{N}$. This subset can be written as
\begin{align*}
    x + \Delta(q\mathbb{N}) = x + (\Delta q)\mathbb{N} = x + p\mathbb{N},
\end{align*}
which we recognise as an arithmetic progression with gap length $p$ in $\mathbb{Z}$.

\begin{lemma}
The set $S$ does not contain any arithmetic progression with integer gap length.
\end{lemma}

\begin{proof}
By way of contradiction, suppose that there is an $x$ in $S$ and a $\Delta > 0$ in $\mathbb{N}$ such that the progression $x + \Delta\mathbb{N}$ is contained in $S$. For the purpose of this argument, we write this progression as an increasing sequence $\{x_n\}_{n = 0}^\infty$. Observe that every term in $x + \Delta\mathbb{N}$ has fractional part in the same $Q_j$. Indeed, every term of this progression has fractional part $\langle x \rangle$. Let $k$ be an index large enough so that 
\begin{enumerate}[(i)]
    \item $\beta_k > x$,
    \item $\beta_{k+1} - \beta_k > \Delta$, and
    \item $k \equiv j \ (\operatorname{mod} N)$.
\end{enumerate}
Such a choice of $k$ is always possible since the sequences $\{\beta_k\}_{k = 0}^\infty$ and $\{\beta_{k+1} - \beta_k\}_{k = 0}^\infty$ are both increasing and unbounded. 

Choose an $n$ so that $x_n$ is the largest term of $x + \Delta\mathbb{N}$ with $x_n < \beta_k$. Condition (i) guarantees that such a term exists. From condition (ii), we know that the next term in the progression, $x_{n + 1} = x_n + \Delta$, must belong to the interval $[\beta_k, \beta_{k+1})$. However, condition (iii) says that $S \cap [\beta_k, \beta_{k+1})$ is a union of translates of $R_j$. Therefore, no element of $S$ belonging to $[\beta_k, \beta_{k+1})$ has a fractional part in $Q_j$, and no term of the arithmetic progression $x+\Delta\mathbb{N}$ can lie in this interval. This is a contradiction.
\end{proof}

\subsection{Avoiding arithmetic progressions with irrational gap length.} 

It remains to show that $S$ does not contain any arithmetic progression where the gap length is an irrational number.

\begin{lemma}
The set $S$ does not contain any arithmetic progression with irrational gap length.
\end{lemma}

\begin{proof}
Suppose for contradiction that there exists an $x$ in $S$ and a $\Delta > 0$ in $\mathbb{R} \setminus \mathbb{Q}$ such that $S$ contains the arithmetic progression $x + \Delta\mathbb{N}$. As above, we write this progression as an increasing sequence $\{x_n\}_{n = 0}^\infty$. We also assume that $x_0 > ((2N+1)/N)\Delta$. (If this condition is not true, apply the argument below to the subsequence $\{x_n\}_{n=3}^\infty$. If $\{x_n\}_{n=3}^\infty$ is not contained in $S$, then the arithmetic progression $\{x_n\}_{n=0}^\infty$ cannot be in $S$ either.) This last condition will help simplify a later estimate; see \eqref{eq:assumption}.

By Weyl's polynomial equidistribution theorem \cite{Weyl}, the sequence of {fractional parts} $\{\langle x_n \rangle\}_{n = 0}^\infty$ is equidistributed in $[0,1)$. This means that for every subinterval $[a,b)$ of $[0,1)$,
\begin{align*}
    \lim_{M \to \infty} \frac{\#(\{\langle x_n \rangle\}_{n=0}^M \cap [a,b))}{M+1} = b - a.
\end{align*}
In particular, this must be true for the subinterval $Q_0$ of $[0,1)$. Fix an arbitrary $\epsilon > 0$. Then there exists an $L$ in $\mathbb{N}$ with the property that whenever $M > L$,
\begin{align}\label{eq:equidistribution}
    \left\vert \frac{\#(\{\langle x_n \rangle\}_{n=0}^M \cap Q_0)}{M+1} - \frac{1}{N}\right\vert < \epsilon.
\end{align}

To achieve a contradiction, fix $\epsilon < 1 / (N^2 + N)$. We exhibit an $M > L$ for which the above inequality fails.

Let $k$ be an index large enough so that 
\begin{enumerate}[(i)]
    \item $\beta_k > x_{L + 1}$,
    \item $k \equiv 1 \ (\operatorname{mod} N)$, and
    \item $(N/(N+1))\beta_k > \Delta$.
\end{enumerate}
Then, let \(M\) be the {largest} natural number such that \(x_M < \beta_k\). Notice that we have \(M \geq L+1 > L\) by condition (i), so \eqref{eq:equidistribution} holds with this choice of $M$.

We make the following observations, to be proved below.

\textit{Claim 1:} The fraction of the terms of $\{x_0, \ldots, x_M\}$ which lie in the interval 
    \begin{align*}
        I = \left[\frac{1}{N + 1}\beta_k, \beta_k\right)
    \end{align*}
is at least $N/(N + 1)$.

\textit{Claim 2:} If $x$ is a point in $S\cap I$, its fractional part is not in $Q_0$.

Together, these claims imply that $\#(\{\langle x_n \rangle\}_{n=0}^M \cap Q_0)$ is at most $(M+1)/(N+1)$. This contradicts \eqref{eq:equidistribution} whenever $\epsilon < 1/(N^2 + N)$.

To prove Claim 1, observe that for any positive real number $a$, the number of elements of the arithmetic progression $\{x_n\}_{n=0}^\infty$ that are contained in the interval \([0,a)\) is given by \(\lceil(a-x_0)/{\Delta}\rceil\). Therefore, the number of terms in the interval $[1/(N+1)\beta_k, \beta_k)$ is 
$$\left\lceil\frac{\beta_k-x_0}{\Delta}\right\rceil-\left\lceil\frac{(1/(N+1))\beta_k-x_0}{\Delta}\right\rceil.$$

Notice also that by choice of $M$ there are $M+1$ terms of $\{x_n\}_{n=0}^\infty$ in $[0, \beta_k)$. It will be convenient to write this as $\left\lceil{(\beta_k-x_0)}/{\Delta}\right\rceil$. Hence, we can express the fraction of the terms of \(\{x_0, \ldots, x_M\}\) that lie in the interval $I$ as
\begin{align}\label{eq:proportion}
    \frac{\left\lceil\frac{\beta_k-x_0}{\Delta}\right\rceil-\left\lceil\frac{(1/(N+1))\beta_k-x_0}{\Delta}\right\rceil}{\left\lceil\frac{\beta_k-x_0}{\Delta}\right\rceil}.
\end{align}
We show that, by our choice of $k$, the above quantity is at least $N/(N+1)$. Using the identity \(a \leq \lceil a\rceil < a+1\), we see that
\begin{align*}
    \left\lceil\frac{\beta_k-x_0}{\Delta}\right\rceil-\left\lceil\frac{\frac{1}{N+1}\beta_k-x_0}{\Delta}\right\rceil &> \frac{\frac{N}{N+1}\beta_k - \Delta}{\Delta}
\end{align*}
is a lower bound for the numerator of \eqref{eq:proportion} and that
\begin{align*}
    \left\lceil\frac{\beta_k-x_0}{\Delta}\right\rceil < \frac{\beta_k-x_0 + \Delta}{\Delta}
\end{align*}
is an upper bound for the denominator of \eqref{eq:proportion}. (Note that both of these bounds are positive due to conditions (iii) and (i) on $k$.) Together, these estimates imply that
\begin{align}\label{eq:assumption}
    \frac{\left\lceil\frac{\beta_k-x_0}{\Delta}\right\rceil-\left\lceil\frac{(1/(N+1))\beta_k-x_0}{\Delta}\right\rceil}{\left\lceil\frac{\beta_k-x_0}{\Delta}\right\rceil} &> \frac{N}{N+1}\left(\frac{\beta_k-\frac{N+1}{N}\Delta}{\beta_k - x_0 +\Delta}\right).
\end{align}
Using our assumption that $x_0 > ((2N+1)/N)\Delta$, we see that the fraction in parentheses above is greater than $1$. We conclude that
\begin{align*}
    \frac{\left\lceil\frac{\beta_k-x_0}{\Delta}\right\rceil-\left\lceil\frac{(\beta_k/(N+1))-x_0}{\Delta}\right\rceil}{\left\lceil\frac{\beta_k-x_0}{\Delta}\right\rceil} >\frac{N}{N+1}.
\end{align*}

Let us now prove Claim 2. Since the index $k$ satisfies $k \equiv 1 \ (\operatorname{mod} N)$, we must also have $k - 1 \equiv 0 \ (\operatorname{mod} N)$. This means that $S \cap [\beta_{k-1},\beta_k)$ is a union of integer translates of $R_0$. Therefore, no element of $S\cap [\beta_{k-1},\beta_k)$ has fractional part in $Q_0$. To prove the claim, it suffices to show that $\beta_k/(N+1) > \beta_{k - 1}$. This will imply that the interval $I$ is a subset of $[\beta_{k-1},\beta_k)$. This is a computation:
\begin{equation*}
    \frac{\beta_k}{N+1} = \sum_{j=0}^{k-1} (N+1)^{j-1} = \frac{1}{N+1} + \sum_{j=0}^{k-2} (N+1)^j = \frac{1}{N+1} + \beta_{k-1} >\beta_{k-1}.  \qedhere
\end{equation*}
\end{proof}

\section{Solution to Repka's problem and higher dimensions}

Let us now explain how Theorem \ref{th:mainresult} provides a positive answer to Repka's problem in the real line. We also show that it implies an analogous result in higher dimensions.

\begin{corollary}\label{cor:reduction}
For each $\lambda$ in $[0,1)$, there exists a subset $S = S(\lambda)$ of $\mathbb{R}$ that intersects every interval of unit length in a set of measure at least $\lambda$, but that does not contain any arithmetic progression.
\end{corollary}

\begin{proof}
Fix $\lambda$ in $[0,1)$. Choose an integer $N$ large enough that $2/N \leq 1 - \lambda$, and apply Theorem \ref{th:mainresult} to obtain a set $S$ (depending on $N$, and hence on $\lambda$) which does not contain any arithmetic progressions, and with the property that
\begin{align}\label{eq:msrpropertyofthm1}
    \vert S \cap [m,m+1] \vert = 1 - \frac{1}{N}
\end{align}
for every $m$ in $\mathbb{Z}$. We claim that this $S$ intersects every interval of unit length in a set of measure at least $\lambda$. 

Let $I$ be an arbitrary interval of unit length in $\mathbb{R}$, and choose an integer $m$ so that $I$ is contained in the union $[m, m+1] \cup [m + 1, m+2]$. Then 
\begin{align*}
        \vert I \setminus S \vert &\leq \vert [m,m+2]\setminus S \vert \\
        &= \vert  [m,m+1] \setminus S \vert + \vert [m+1, m+2] \setminus S \vert \\
        &= \frac{1}{N} + \frac{1}{N} \\
        &\leq 1 - \lambda,
\end{align*}
where we have used \eqref{eq:msrpropertyofthm1} in the third line. This implies that $S \cap I$ has measure at least $\lambda$.
\end{proof}

After submitting this article, we were informed by an anonymous referee that it suffices to construct a set that does not contain arithmetic progressions with irrational gap length. With the referee's permission, we include their argument below. While this argument is shorter than the construction in Section \ref{sec:proof}, the latter has the advantage of producing a more explicit set, in the sense that intersections with any unit interval are easily illustrated.

\begin{proposition}
Suppose that for each $\mu$ in $[0,1)$, there exists a subset $S$ of $\mathbb{R}$ that intersects every interval of unit length in a set of measure at least $\mu$, but that does not contain any arithmetic progression with irrational gap length. Then, for each $\lambda$ in $[0,1)$, there exists a subset $T$ of $\mathbb{R}$ that intersects every interval of unit length in a set of measure at least $\lambda$, but that does not contain any arithmetic progression with any gap length.
\end{proposition}

\begin{proof}
    Given $\lambda$ in $[0,1)$, choose any irrational number $r > 1$ and choose $\mu$ with $(\lambda + r)/(1 + r) \leq \mu < 1$. By assumption, we know that there exists a set $S$ that intersects every interval of unit length in a set of measure at least $\mu$, but that does not contain any arithmetic progression with irrational gap length.

    Let $T = S \cap rS$. Fix an interval of unit length $I$, and choose an interval of unit length $J$ such that $r^{-1}I \subset J$. Then
    \begin{align*}
        \vert I \setminus T \vert &\leq \vert I \setminus S \vert + \vert I \setminus rS \vert \\ 
        &= \vert I \setminus S \vert + r \vert (r^{-1} I) \setminus S \vert \\
        &\leq \vert I \setminus S \vert + r \vert J \setminus S \vert \\
        &\leq 1 - \mu + r (1 - \mu) \\
        &\leq 1 - \lambda.
    \end{align*}
    This implies that $T \cap I$ has measure at least $\lambda$.
    
    Since $T \subseteq S$, $T$ does not contain any arithmetic progression with irrational gap length. Suppose for contradiction that there exist $x \in \mathbb{R}$ and $\Delta \in \mathbb{Q}$ such that the arithmetic progression $x + \Delta \mathbb{N}$ is contained in $T$. Then $x + \Delta \mathbb{N}$ must also belong to $rS$. This implies that the arithmetic progression $(x/r) + (\Delta/r)\mathbb{N}$ with irrational gap length $\Delta/r$ belongs to $S$, which contradicts our choice of $S$. Thus, $T$ contains no arithmetic progression.    
\end{proof}

The next corollary to Theorem \ref{th:mainresult} asserts that by taking $n$-fold Cartesian products of the sets $S$ constructed in the proof of Corollary \ref{cor:reduction}, we obtain large subsets of $n$-dimensional Euclidean space that avoid arithmetic progressions.

\begin{corollary}
For each $\lambda$ in $(0,1)$, and each $n \geq 1$, there exists a subset $S^n = S^n(\lambda, n)$ of $\mathbb{R}^n$ that intersects every cube (with sides parallel to the axes) of unit volume in a set of measure at least $\lambda$, but that does not contain any arithmetic progression.
\end{corollary}

\begin{proof}
Observe that if $A \subset \mathbb{R}^m$ and $B \subset \mathbb{R}^n$ do not contain arithmetic progressions, then their Cartesian product $A \times B \subset \mathbb{R}^{m+n}$ also does not contain arithmetic progressions. We prove the contrapositive. Suppose that there is an $x = (x_1, \ldots, x_m, x_{m+1}, \ldots, x_{m+n})$ in $A \times B$ and a nonzero $\Delta = (\delta_1, \ldots, \delta_{m+n})$ in $\mathbb{R}^{m+n}$ such that
\begin{align*}
    x + \Delta\mathbb{N} = \{(x_1 + k\delta_1, \ldots, x_m + k\delta_m, x_{m+1} + k\delta_{m+1}, \ldots, x_{m+n} + k\delta_{m+n}) : k \in \mathbb{N}\}
\end{align*}
is contained in $A \times B$. 

Consider the projection maps $\pi_A : A \times B \to A$ and $\pi_B : A \times B \to B$. Since $\Delta \neq 0$, we must have that either $\pi_A(\Delta) = (\delta_1, \ldots, \delta_m) \neq 0$ or $\pi_B(\Delta) = (\delta_{m+1}, \ldots, \delta_{m+n}) \neq 0$. Without loss of generality, suppose that $\pi_A(\Delta) \neq 0$. Then the image $\pi_A(x + \Delta\mathbb{N})$ in $A$ is the set
\begin{align*}
    \pi_A(x + \Delta\mathbb{N}) = \{(x_1 + k\delta_1, \ldots, x_m + k\delta_m) : k \in \mathbb{N}\}.
\end{align*}
We recognize this as the arithmetic progression $\pi_A(x) + \pi_A(\Delta) \mathbb{N}$ with nonzero gap $\pi_A(\Delta)$. This progression is contained in $A$.

Now fix $\lambda$ in $(0,1)$, and consider a cube $Q \subset \mathbb{R}^n$ with unit volume. For each $i$, let $\pi_i : \mathbb{R}^n \to \mathbb{R}$ be the projection map onto the $i^\text{th}$ coordinate axis. Then the sets $\pi_1(Q), \ldots, \pi_n(Q)$ are intervals of unit length in $\mathbb{R}$. 

Apply Corollary \ref{cor:reduction} with $\lambda^{1/n}$, to obtain a subset $S$ of $\mathbb{R}$ that does not contain any arithmetic progression, and with $\vert S \cap \pi_i(Q) \vert \geq \lambda^{1/n}$ for each $i$.

Let $S^n$ be the $n$-fold Cartesian product $S^n = S \times \ldots \times S$. From our above observations, we know that $S^n$ does not contain any arithmetic progression. Also, by definition of product measures, we have that
\begin{align*}
    \vert S^n \cap Q\vert = \vert (S \cap \pi_1(Q)) \times \ldots \times (S \cap \pi_n(Q)) \vert = \prod_{i=1}^n \vert S \cap \pi_i(Q) \vert \geq \lambda,
\end{align*}
as required.

(Note that we use $\vert \cdot \vert$ to denote both the $1$-dimensional and $n$-dimensional Lebesgue measure.)
\end{proof}

\section{The range of $\lambda$ is optimal}

If a subset $S$ of $\mathbb{R}$ intersects every interval of unit length in a set of \textit{full} measure, then the complement of $S$ in $\mathbb{R}$ must have finite measure. It is known, however, that any subset of $\mathbb{R}$ whose complement has finite measure must contain an arithmetic progression. Thus, Corollary \ref{cor:reduction} fails when $\lambda = 1$.

The fact that sets whose complements have finite measure must contain arithmetic progressions follows Theorem \ref{th:chlebik}, below. This result is due to Chlebík, as reported by Svetic in 2000 \cite{Svetic}. A proof appears in a 2015 preprint by Chleb\'ik \cite{Chlebik}.

Chleb\'ik terms a subset $X \subseteq \mathbb{R}$ \textit{uniformly locally finite} if 
\begin{align*}
    \operatorname{sup}\{\#(X \cap [u,u+1]): u \in \mathbb{R}\} < \infty.
\end{align*}

\begin{externaltheorem}[Chleb\'ik, Theorem 14(a) from \cite{Chlebik}]\label{th:chlebik}
    Let $X \subseteq \mathbb{R}$. If $X$ is uniformly locally finite then there is a constant $\epsilon > 0$ such that whenever $G \subseteq \mathbb{R}$ is a Lebesgue measurable set with $|G| < \epsilon$, then $\mathbb{R} \setminus G$ contains plenty of translations of $X$; namely $\{a \in \mathbb{R} : (a + X) \subseteq (\mathbb{R} \setminus G)\}$ has infinite Lebesgue measure.
\end{externaltheorem}

This result implies Proposition \ref{prop:finitecomplement} when $X = \mathbb{N}$ and $G = [\epsilon/(2\vert \mathbb{R}\setminus S \vert)](\mathbb{R} \setminus S)$. For convenience, we include below an elementary proof of the relevant case of Chleb\'ik's result.

\begin{proposition}\label{prop:finitecomplement}
If the complement of a subset $S$ of $\mathbb{R}$ has finite measure, then $S$ contains a two-sided arithmetic progression $x + \Delta \mathbb{Z}$ for some $x$ in $S$ and nonzero $\Delta$ in $\mathbb{R}$.
\end{proposition}

\begin{proof} Let $S$ be a subset of $\mathbb{R}$ whose complement has finite measure. Then $\vert \mathbb{R} \setminus S \vert < \xi < \infty$ for some positive real number $\xi$. For each integer $k$, let $I_k$ denote the open interval $(2k\xi, 2(k+1)\xi)$. Let $\tau_{y}:\mathbb{R} \rightarrow \mathbb{R}$ be the translation map defined by $\tau_y(x) = x+y$, and consider the set
\begin{align}\label{eq:setI}
    I = \bigcap_{k \in \mathbb{Z}} \tau_{-2k\xi}(S \cap I_k).
\end{align}

To show that $S$ contains an infinite arithmetic progression, it suffices to show that $I$ is nonempty. Indeed, if $x \in I$, then we have $x + 2k\xi$ contained in $S \cap I_k$ for each $k \in \mathbb{Z}$, and hence the arithmetic progression $x + 2\xi\mathbb{Z}$ is contained in $S$.

We prove that $I$ is nonempty by showing that it has positive measure. We require the fact that $I \subset I_0$, as well as the subadditivity and translation invariance of Lebesgue measure. By the definition of $I$ in \eqref{eq:setI},
\begin{align*}
|I| &= \left|\bigcap_{k \in \mathbb{Z}} \tau_{-2k\xi}(S \cap I_k)\right|.
\end{align*}
Since $I \subset I_0$, we may write
\begin{align*}
|I| &=\left|I_0 \setminus \bigcup_{k \in Z} \left(I_0 \setminus \tau_{-2k\xi}(S \cap I_k)\right)\right|.
\end{align*}
Now, since $I_0 \setminus \tau_{-2k\xi}(S \cap I_k)$ is contained in $I_0$, and since $I_0$ has finite measure, we see that
\begin{align*}
|I| &=|I_0| - \left|\bigcup_{k \in Z} \left(I_0 \setminus \tau_{-2k\xi}(S \cap I_k)\right)\right|.
\end{align*}
From the subadditivity of measures, we obtain
\begin{align*}
|I| &\geq |I_0| - \sum_{k \in \mathbb{Z}} \left|I_0 \setminus \tau_{-2k\xi}(S \cap I_k)\right|.
\end{align*}
Translating the intervals $I_0 \setminus \tau_{-2k\xi}(S \cap I_k)$ forward by $2k\xi$, and appealing to the translation invariance of the Lebesgue measure, we see that
\begin{align*}
|I| &\geq |I_0| - \sum_{k \in \mathbb{Z}} \left|\tau_{2k\xi}(I_0) \setminus (S \cap I_k)\right|.
\end{align*}
Recognising $\tau_{2k\xi}(I_0)$ as the interval $I_k$, we may write
\begin{align*}
|I| &\geq|I_0| - \sum_{k \in \mathbb{Z}} \left|I_k \setminus (S \cap I_k)\right|.
\end{align*}
The identity $A\setminus (B\cap C) = (A\setminus B)\cup (A\setminus C)$ now gives
\begin{align*}
|I| &\geq |I_0| - \sum_{k \in \mathbb{Z}} \left|I_k \setminus S\right|,
\end{align*}
and by the countable additivity of measures, we have
\begin{align*}
|I| &\geq |I_0| - |\mathbb{R} \setminus S|.
\end{align*}
By definition, $I_0$ has Lebesgue measure $2\xi$. Moreover, the complement of $S$ in $\mathbb{R}$ has measure less than $\xi$, so
\begin{align*}
|I| &> \xi > 0
\end{align*}
as needed.
\end{proof}

Although Corollary 3.1 fails when $\lambda = 1$, it is reasonable to ask the following weaker question: does there exist an increasing sequence of positive real numbers $\{\lambda_m\} \to 1^{-}$ in $[0,1)$ and a subset $S$ of $\mathbb{R}$ satisfying $\vert S \cap [m,m+1] \vert \geq \lambda_m$ for each $m$, such that $S$ does not contain any arithmetic progression? 

The following theorem answers this question affirmatively.

\begin{externaltheorem}[Burgin--Goldberg--Keleti--MacMahon--Wang, Theorem 1.3 from \cite{newkeleti}]
    There exists a closed set $S \subset [0,\infty)$ satisfying $\lim_{m \to \infty} \vert S \cap [m,m+1] \vert = 1$ such that $S$ does not contain any infinite arithmetic progression.
\end{externaltheorem}

\section*{Acknowledgements}

The authors thank J. Repka for proposing the problem addressed in this paper and M. Pramanik for many helpful comments regarding the exposition. We also thank the anonymous referee for their valuable comments and for suggesting Proposition 3.2.

\bibliographystyle{amsalpha}

\end{document}